\documentclass[11pt]{amsart}
\usepackage{hyperref}

\newcommand{\R}{\mathbf{R}}

\newcommand{\N}{\mathbf{N}}
\newcommand{\tto}{\Rightarrow}
\newcommand{\eps}{\epsilon}

\theoremstyle{plain}
\newtheorem{theorem}{Theorem}%[section]
\newtheorem{proposition}[theorem]{Proposition}
\newtheorem{lemma}[theorem]{Lemma}

\theoremstyle{definition}

\theoremstyle{remark}

\begin{document}
\title[New simple proofs of theorems of Kolmogorov and Prokhorov]
{New simple proofs of the Kolmogorov extension theorem and Prokhorov's theorem}
\author{Wooyoung Chin}
\maketitle

\begin{abstract}
We provide new simple proofs of the Kolmogorov extension theorem
and Prokhorovs' theorem.
The proof of the Kolmogorov extension theorem is based on the simple
observation that $\R$ and the product measurable space $\{0,1\}^\N$
are Borel isomorphic.
To show Prokhorov's theorem, we observe that we can assume that the
underlying space is $\R^\N$.
Then the proof of Prokhorov's theorem is a straightforward application
of the Kolmogorov extension theorem we just proved.
\end{abstract}

%In this note, we provide a new simple proof of each of the following
%important measure-theoretic theorems in probability theory.

The Kolmogorov extension theorem (Theorem \ref{thm:kolmogorov_extension})
is useful when showing the existence of a stochastic process whose
finite-dimensional distributions are known.
By a \emph{finite dimensional distribution} of a process $(X_t)_{t\in T}$,
we mean the joint density of $X_{t_1},\ldots,X_{t_n}$ for some
$n \in \N$ and $t_1,\ldots,t_n \in T$.
For example, one can construct a Brownian motion \cite[Section 37]{Bil12}
or a Markov chain \cite[Section 5.2]{Dur19} using the Kolmogorov
extension theorem.

In this note, we provide a proof of the Kolmogorov extension theorem
based on the simple, but perhaps not widely known observation that
$\R$ and the product measurable space $2^\N$ are Borel isomorphic.
(We denote by $2$ the discrete space $\{0,1\}$.)
By a \emph{Borel isomorphism} we mean a measurable bijection whose inverse is
also measurable.

%%% Use later
%Note that the product $\sigma$-field on $2^\N$ coincides with the
%Borel $\sigma$-field on $2^\N$ generated by the product topology on $2^\N$.
%The following is the version of the Kolmogorov extension theorem we prove.

%Since any Borel measurable subset of $\R^T$ depends only on countably many
%coordinates, it turns out that we can only consider the case $T = \N$.
%For example, if $(B_t)_{t\in[0,\infty)}$ is a Brownian motion,
%we know the joint distribution of $B_{t_1},\ldots,B_{t_n}$ ($t_1<\cdots<t_n$)
%since $B_{t_1},B_{t_2}-B_{t_1},\ldots,B_{t_n}-B_{t_{n-1}}$ are
%independent normal random variables with mean $0$ and variances
%$t_1,t_2-t_1,\ldots,t_n-t_{n-1}$.
%We can apply the Kolmogorov extension theorem to show the existence of
%a process $(B_t)_{t\in[0,\infty)}$ with the properties (i) $B_t=0$ a.s.,
%(ii) if $t_1<\cdots<t_n$, then $X_{t_1},X_{t_2}-X_{t_1},\ldots,
%X_{t_n}-X_{t_{n-1}}$

\begin{theorem}[Kolmogorov extension theorem] \label{thm:kolmogorov_extension}
For each $n \in \N$, let $\mu_n$ be a Borel probability measure on $\R^n$.
Assume that $\mu_1,\mu_2,\ldots$ are \emph{consistent} in the sense that
$\mu_{n+1}(A \times \R) = \mu_n(A)$ for any $n \in \N$ and any Borel set
$A \subset \R^n$.
Then there exists a Borel probability measure $\mu$ on $\R^\N$ satisfying
$\mu(A \times \R^\N) = \mu_n(A)$ for any $n \in \N$ and any Borel set
$A \subset \R^n$.
\end{theorem}

Prokhorov's theorem (Theorem \ref{thm:prokhorov})
is useful when one wants to show that a certain sequence of random
functions converges in distribution to some limiting random function.
For example, it can be used to show that if $\xi_1,\xi_2,\ldots$ are i.i.d.\
real-valued random variables with mean $0$ and variance $1$, then
the sequence of random functions $X_n\colon[0,1]\to\R$ given by
\[
X_n(t) = \frac{\xi_1 + \cdots + \xi_{\lfloor nt\rfloor}
+ (nt-\lfloor nt\rfloor)\xi_{\lfloor nt\rfloor+1}
}{\sqrt{n}}
\]
converges in distribution to a Brownian motion from time $0$ to $1$.

A sequence $(\mu_n)_{n\in\N}$ of probability measures on a metric space $S$
is said to be \emph{tight} if for each $\eps>0$ there is a compact $K\subset S$
such that $\mu_n(K) \ge 1-\eps$ for all $n\in\N$.
In the second half of the note, we will prove the following theorem.
The space $S$ we have in mind is a function space,
such as the space of all continuous real-valued functions on $[0,1]$.

\begin{theorem}[Prokhorov's theorem] \label{thm:prokhorov}
Any tight sequence of probability measures on a separable metric space $S$
has a weakly convergent subsequence.
\end{theorem}

%The Kolmogorov extension theorem for possibly uncountable products of
%Borel spaces can be easily derived from Theorem \ref{thm:kolmogorov_extension},
%but we omit the derivation.
%Our proof of Theorem \ref{thm:kolmogorov_extension} is based on the
%observation that $\R$ is isomorphic to the Cantor space $2^\N$
%(where $2$ denotes the discrete space $\{0,1\}$ as measurable spaces).

The following special case \cite[Theorem 29.3]{Bil12}
of Theorem \ref{thm:prokhorov} can be easily proved by a diagonalization argument
applied to (cumulative) distribution functions.

\begin{theorem} \label{thm:prokhorov_euclidean}
Any tight sequence of probability measures on $\R^n$ ($n \in \N$)
has a weakly convergent subsequence.
\end{theorem}

Unlike Theorem \ref{thm:prokhorov_euclidean}, the existing proofs of
Theorem \ref{thm:prokhorov} in the literature are rather involved.
For example, a proof in \cite[Section 5]{Bil99} (which does not assume the separability of the underlying metric space) uses the Carath\'eodory extension
theorem \cite[Theorem 11.1]{Bil12}, but showing that the constructed outer
measure is indeed an outer measure satisfying some desired properties
is nontrivial.

Our proof of Theorem \ref{thm:prokhorov} uses Theorem
\ref{thm:kolmogorov_extension} to first cover the case $S=\R^\N$.
Although one might consider Theorem \ref{thm:kolmogorov_extension}
as a nontrivial tool, we will see that it is more accessible than the tools
used by other proofs.
We then use the fact that any separable metric space can be topologically
embedded in $\R^\N$, which is a fact also used by some other proofs,
to generalize the result to an arbitrary separable metric space $S$.

\section{The Kolmogorov extension theorem} \label{sec:kolmogorov_extension}

The key to simplifying Theorem \ref{thm:kolmogorov_extension}
is to observe that we can replace $\R$ (equipped with the Borel $\sigma$-field)
with the product space $2^\N$.
Note that the product $\sigma$-field on $2^\N$ coincides with the
Borel $\sigma$-field on $2^\N$ generated by the product topology on $2^\N$.

Let $C$ be the set of all sequences in $2^\N$ which are eventually constant,
and $D$ be the set of all dyadic rationals in $(0,1)$.
Then, the function $f\colon 2^\N\setminus C \to (0,1)\setminus D$
given by
\[
f\bigl((a_n)_{n \in \N}\bigr) = \sum_{n=1}^\infty a_n2^{-n}
\]
is a measurable bijection.
Since $C$ and $D$ are both countably infinite, there is a bijection
$g\colon C\to D$.
Let $h\colon 2^\N \to (0,1)$ be given by $h(x)=f(x)$ if $x\in 2^N\setminus C$
and $h(x)=g(x)$ if $x\in C$.
It is clear that $h$ is a measurable bijection.

\begin{proposition}
%There is a Borel isomorphism between $2^\N$ and $\R$,
%where both sets are equipped with the Borel $\sigma$-field.
The inverse $h^{-1}$ is also measurable.
\end{proposition}

\begin{proof}
%Let $C$ be the set of all sequences in $2^\N$ which are eventually constant,
%and $D$ be the set of all dyadic rationals in $(0,1)$.
%Then, the function $f\colon 2^\N\setminus C \to (0,1)\setminus D$
%given by
%\[
%f\bigl((a_n)_{n \in \N}\bigr) = \sum_{n=1}^\infty a_n2^{-n}
%\]
%is a measurable bijection.
Since any subset of $C$ and $D$ are measurable, it is enough to show
that $f(A)$ is measurable for any measurable $A\subset 2^\N\setminus C$.
For any $N\in\N$ and $(b_n)_{1\le n\le N}\in 2^N$, we have
\begin{multline*}
f\bigl(\{\,(a_n)\in 2^\N \mid a_n=b_n \text{ for } n=1,\ldots,N\,\}\bigr) \\
= \Bigl(\sum_{n=1}^N b_n2^{-n},\sum_{n=1}^N b_n2^{-n} + 2^{-N}\Bigr)\setminus D,
\end{multline*}
where the right side is measurable.
Since the collection of all measurable subsets of $2^\N\setminus C$
whose image under $f$ is measurable is a $\lambda$-system,
Dynkin's $\pi$-$\lambda$ theorem \cite[Theorem 3.2]{Bil12} finishes the proof.
%Since $C$ and $D$ are both countably infinite, there is a bijection
%$g\colon C\to D$.
%If we define $h\colon 2^\N \to (0,1)$ by $h(x)=f(x)$ if $x\in 2^N\setminus C$
%and $h(x)=g(x)$ if $x\in C$, then $h$ is a Borel isomorphism.
%Since $(0,1)$ is homeomorphic to $\R$, the proof is finished.
\end{proof}

\begin{proof}[Proof of Theorem \ref{thm:kolmogorov_extension}]
Let $[n] := \{1,\ldots,n\}$ for each $n \in \N$.
By replacing $\R$ with $2^\N$, we can now assume that $\mu_n$ is
defined on $2^{\N \times [n]}$ instead.
Our goal is to find a $\mu$ defined on $2^{\N \times \N}$
which is consistent with $\mu_1, \mu_2, \ldots$.
For any sets $F,G$ with $F \subset G$, let
$\pi_{G,F} \colon 2^G \to 2^F$ denote the canonical projection.
For a fixed $G$, let us call a set of the form $\pi_{G,F}^{-1}(A)$, where
$F \subset G$ is finite and $A \subset 2^F$, a \emph{cylinder subset} of $2^G$.
For each $n \in \N$ and a cylinder subset $A$ of $2^{\N \times [n]}$,
define
\[
\mu(\pi_{\N\times\N,\N\times[n]}^{-1}(A)) := \mu_n(A).
\]
Then, by the consistency and the finite additivity of the $\mu_n$'s,
$\mu$ is a well-defined and finitely additive set function on the field of
cylinder subsets of $2^{\N \times \N}$.

Since any cylinder subset of $2^{\N \times \N}$ is compact and open,
whenever a cylinder subset is the disjoint union of other cylinder subsets,
the union must be finite.
Thus $\mu$ is obviously (or vacuously) countably additive, and so
the Hahn-Kolmogorov theorem \cite[Theorem 11.2]{Bil12} (sometimes called
the Carath\'eodory extension theorem) implies that $\mu$ extends to
a Borel probability measure, also denoted by $\mu$, on $2^{\N \times \N}$.

It remains to show that $\mu$ is consistent with the $\mu_n$'s.
For any $n \in \N$, the collection of Borel sets $A \subset 2^{\N \times [n]}$
for which
\begin{equation} \label{eqn:kolmogorov_pf}
\mu(\pi_{\N\times\N,\N\times[n]}^{-1}(A)) = \mu_n(A)
\end{equation}
is a $\lambda$-system containing the cylinder subsets of $2^{\N \times [n]}$.
So, by Dynkin's $\pi$-$\lambda$ theorem, we have \eqref{eqn:kolmogorov_pf}
for all Borel sets $A \subset 2^{\N \times [n]}$.
\end{proof}

\section{Prokhorov's theorem} \label{sec:prokhorov}

%A proof of Theorem \ref{thm:prokhorov} found in \cite[Theorem II.6.4]{Par67}
%makes use of the following well-known fact.

%\begin{theorem} \label{thm:urysohn}
%If $S$ is a separable metric space, then $S$ can be homeomorphically
%embedded into $\R^\N$.
%\end{theorem}
%
%\begin{proof}
%Let $\{\,s_n \mid n\in\N\,\}$ be a countable dense subset of $S$.
%Denoting the metric on $S$ by $d$, define $f\colon S\to[0,1]^\N$ by
%\[
%f(x) := (\min\{d(x,s_n),1\})_{n\in\N}.
%\]
%For any distinct $x_1,x_2\in S$, we can find some $n \in \N$ such that
%$d(x_1,s_n) < d(x_2,s_n)$ and $d(x_1,s_n) < 1$, so that
%$\min\{d(x_1,s_n),1\} < \min\{d(x_2,s_n),1\}$.
%This shows that $f$ is injective.
%The continuity of $f$ follows from the continuity of $d\colon S\times S\to\R$.
%To show that $f\colon S\to f(S)$ is open, let $U\subset S$ be open and
%$y \in f(U)$ be given. Let $x \in S$ be such that $y=f(x)$, and find an $n\in\N$
%such that (i) the open ball of radius $
%\end{proof}

%As we see in the following lemma, Theorem \ref{thm:urysohn} allows us to
%prove Theorem \ref{thm:prokhorov} for only $[0,1]^\N$.

The proof of Theorem \ref{thm:prokhorov} in
\cite[Theorem II.6.4, Theorem II.6.7]{Par67}
starts with the case that $S$ is a compact metric space, using the
Riesz-Markov-Kakutani representation theorem \cite[Theorem 2.14]{Rud87}.
Since any separable metric space can be topologically embedded in
the compact metrizable space $[0,1]^\N$ (see \cite[p. 125]{Kel55} or
the proof of \cite[Theorem 34.1]{Mun00}),
%we may assume that $S$
%is a subspace of $[0,1]^\N$ when we prove Theorem \ref{thm:prokhorov}.
%(Note that only the topology, not the specific metric, of the underlying space
%matters if we are concerned with the Borel sets, tightness, and weak
%convergence.
%We should also warn that $S$ might not be measurable in $[0,1]^\N$.)
Theorem \ref{thm:prokhorov} follows from the result of \cite{Par67}
by using the following lemma.
The notation $\tto$ denotes weak convergence.

\begin{lemma} \label{lem:reduction}
Assume that any tight sequence of probability measures on a metric space
$T$ has a weakly convergent subsequence.
If $S \subset T$, then any tight sequence $(\mu_n)_{n\in\N}$
of probability measures on $S$ also has a weakly convergent subsequence.
\end{lemma}

\begin{proof}
If we let
\[
\nu_n(A) := \mu_n(A\cap S) \quad \text{for Borel $A\subset T$,}
\]
then $(\nu_n)_{n\in\N}$ is a tight sequence of probability measures on $T$.
By the assumption, $\nu_{n_k}\tto \nu$ as $k\to\infty$ for some
probability measure $\nu$ on $T$ and $n_1<n_2<\cdots$.

Given $m\in\N$, there is some compact $K_m\subset S$ such that
\[
\nu_{n_k}(K_m) = \mu_{n_k}(K_m) \ge 1-1/m \qquad \text{for all $k\in\N$}.
\]
Let $E:=\bigcup_{m\in\N} K_m$.
As each $K_m$ is closed in $T$, we have
\[
\nu(K_m) \ge \limsup_{k\to\infty} \nu_{n_k}(K_m) \ge 1 - 1/m
\qquad \text{for all $m\in\N$}
\]
by the portmanteau theorem \cite[Theorem 2.1]{Bil99}.
So, $\nu(E)=1$.

Define a probability measure $\mu$ on $S$ by
\[
\mu(A):=\nu(A\cap E).
\]
Assume that $C\subset S$ is closed in $S$.
If $D\subset T$ is a closed set satisfying $D \cap S = C$, then
\[
\begin{split}
\limsup_{k\to\infty} \mu_{n_k}(C) &= \limsup_{k\to\infty} \nu_{n_k}(D)
\le \nu(D) \\
&= \nu(D\cap E) = \nu(C \cap E) = \mu(C).
\end{split}
\]
Thus, $\mu_{n_k} \tto \mu$ as $k\to\infty$ by the portmanteau theorem.
\end{proof}

The original paper by Prokhorov \cite[Theorem 1.12]{Pro56} shows
Theorem \ref{thm:prokhorov} when $S$ is a complete and separable metric space,
by first developing the theory of the Prokhorov metric on the space
of probability measures on $S$.
Since $[0,1]^\N$ is separable and complete under some metric generating
the product topology, we can again derive Theorem \ref{thm:prokhorov}
from Prokhorov's result using Lemma \ref{lem:reduction}.

Below we provide a proof of Theorem \ref{thm:prokhorov}
that works specifically when $S = \R^\N$.
Since any separable metric space can be embedded in $\R^\N$
(because $[0,1]^\N$ is a subspace of $\R^\N$),
Theorem \ref{thm:prokhorov} for arbitrary $S$
would then follow as above.

We use Theorem \ref{thm:kolmogorov_extension} in our proof.
However, as one can see from Section \ref{sec:kolmogorov_extension},
proving Theorem \ref{thm:kolmogorov_extension} takes less effort than proving
the Riesz-Markov-Kakutani representation theorem or developing the theory of
the Prokhorov metric.
Also, Theorem \ref{thm:kolmogorov_extension} can be naturally introduced
in a course in probability theory.
(Note that the Prokhorov metric is covered in a starred section
in \cite{Bil99}.)
When $\mu$ is a measure on $S_1$ and $\pi\colon S_1\to S_2$ is a measurable
function, we denote the measure $A \mapsto \mu(\pi^{-1}(A))$ on $S_2$
by $\mu\pi^{-1}$.

%Compared to the Riesz-Markov-Kakutani representation theorem or
%the Prokhorov metric, Theorem \ref{thm:kolmogorov_extension} also
%fits better in a course in probability theory.
%For example, 
%for example, the Prokhorov metric is contained in a starred section
%in \cite{Bil99}.

% ingredient is the following finite-dimensional version of
%Prokhorov's theorem, which can be proved just as in the one-dimensional case
%by applying a diagonalization argument to distribution functions.

%whose standard proof we contain for completeness.
%By $(x_1,\ldots,x_d) < (y_1,\ldots,y_d)$ we mean $x_i < y_i$ for all
%$i=1,\ldots,d$.

%\begin{proof}
%Let $(\mu_n)_{n \in \N}$ be a tight sequence of probability measures on $\R^d$.
%For each $n \in \N$, let $F_n$ denote the distribution function of $\mu_n$.
%By a diagonalization argument, we can choose $n_1 < n_2 < \cdots$ in $\N$
%such that $F_{n_k}(x)$ converges as $k \to \infty$ for all $x \in \Q^d$.
%Let $G(x)$ be the limit of $F_{n_k}(x)$ as $k \to \infty$.
%Define $F \colon \R^d \to [0,1]$ by
%\[
%F(x) := \inf\{\,G(y) \mid y \in \Q^d, y > x\,\}.
%\]
%Using the tightness of $(\mu_n)_{n \in \N}$, it is not difficult to see that
%$F$ is a distribution function.
%It is also not difficult to see that $F_{n_k}(x) \to F(x)$ as $k \to \infty$
%whenever $F$ is left-continuous at $x$, i.e. whenever for each $\eps>0$
%there is a $y<x$ such that $y<z<x$ implies $|F(z)-F(x)|\le\eps$.
%If $\mu$ is the probability measure on $\R^d$ whose distribution function is $F$,
%then $\mu_{n_k} \tto \mu$ as $k \to \infty$.
%\end{proof}

\begin{theorem} \label{thm:prokhorov_sequence}
Any tight sequence of probability measures on $\R^\N$ has a
weakly convergent subsequence.
\end{theorem}

\begin{proof}
Let $(\mu_n)_{n \in \N}$ be a given tight sequence of probability measures
on $\R^\N$.
For each $N \in \N$, let $\pi_N\colon \R^\N \to \R^N$ denote the projection
to the first $N$ coordinates.
Since $\mu_n(K) \ge 1-\eps$ implies
\[
\mu_n\pi_N^{-1}(\pi_N(K)) \ge 1-\eps,
\]
the sequence $(\mu_n\pi_N^{-1})_{n \in \N}$ is tight for all $N \in \N$.
Thus, by Theorem \ref{thm:prokhorov_euclidean} and a diagonization argument,
we have $n_1 < n_2 < \cdots$ such that $\mu_{n_k}\pi_N^{-1} \tto \nu_N$
as $k \to \infty$ (for some $\nu_N$) for all $N \in \N$.

Let $\pi_{N+1,N}\colon \R^{N+1} \to \R^N$ be the projection to the first
$N$ coordinates.
Since $\mu_{n_k}\pi_{N+1}^{-1} \tto \nu_{N+1}$ as $k \to \infty$, we have
\[
\mu_{n_k}\pi_N^{-1} = (\mu_{n_k}\pi_{N+1}^{-1})\pi_{N+1,N}^{-1}
\tto \nu_{N+1}\pi_{N+1,N}^{-1}
\]
as $k \to \infty$ by the definition of weak convergence
(or the continuous mapping theorem \cite[Theorem 2.7]{Bil99}).
This shows $\nu_{N+1}\pi_{N+1,N}^{-1} = \nu_N$ for all $N \in \N$.
Thus, Theorem \ref{thm:kolmogorov_extension} implies the existence
of a probability measure $\nu$ on $\R^\N$ satisfying
$\nu\pi_N^{-1} = \nu_N$ for all $N \in \N$.

To show $\mu_{n_k} \tto \nu$ as $k \to \infty$, let $G \subset \R^\N$
be a given open set.
Since any open subset of $\R^\N$ is a countable union of sets of the form
$\pi_M^{-1}(H)$ where $M\in\N$ and $H\subset\R^M$ is open, for any $\eps > 0$
there is some $M \in \N$ and an open set $H \subset \R^M$ such that
\[
\pi_M^{-1}(H) \subset G\quad\text{and}\quad\nu(\pi_M^{-1}(H)) \ge \nu(G) - \eps.
\]
Since $\mu_{n_k}\pi_M^{-1} \tto \nu\pi_M^{-1}$ as $k \to \infty$, we have
\[
\nu(G)-\eps \le \nu\pi_M^{-1}(H) \le \liminf_{k\to\infty}\mu_{n_k}\pi_M^{-1}(H)
\le \liminf_{k\to\infty}\mu_{n_k}(G).
\]
Since $G$ and $\eps$ were arbitrary, we can conclude that
$\mu_{n_k} \tto \nu$ as $k \to \infty$
by the portmanteau theorem.
\end{proof}

% References
\bibliographystyle{alpha}
\bibliography{wooyoung}
\end{document}